\documentclass[11pt, a4paper]{amsart}

\pdfoutput=1

\usepackage{amsmath,amsthm,bbm,mlmodern}
\usepackage{amssymb}
\usepackage{hyperref}
\usepackage{microtype}
\usepackage{graphicx}

\newtheorem{theorem}{Theorem}
\newtheorem{proposition}{Proposition}
\newtheorem{lemma}{Lemma}

\theoremstyle{definition}

\newtheorem{assumption}{Assumption}
\newtheorem{remark}{Remark}

\newcommand{\norm}[1]{\|#1\|}
\newcommand{\supnorm}[1]{\norm{#1}_{\infty}}
\newcommand{\lipnorm}[1]{\norm{#1}_{\textup{Lip}}}
\newcommand{\bvnorm}[1]{\norm{#1}_{\textup{BV}}}
\newcommand{\lnorm}[2]{\norm{#1}_{L^{#2}}}
\newcommand{\pcal}{\mathcal{P}}
\newcommand{\naturals}{\mathbb{N}}
\newcommand{\reals}{\mathbb{R}}
\newcommand{\torus}{\mathbbm{T}}
\newcommand{\bigo}{O}
\newcommand{\charfun}{\mathbbm{1}}
\newcommand{\diff}{\, \mathrm{d}}
\newcommand{\haus}{\mathcal{H}}
\newcommand{\leb}{m}
\newcommand{\geqs}{\gtrsim}

\DeclareMathOperator{\dist}{dist}
\DeclareMathOperator{\supp}{supp}
\DeclareMathOperator{\var}{var}

\title{Strong Borel--Cantelli lemmas for recurrence}

\author{Tomas Persson}
\address{Centre for Mathematical Sciences,
Lund University, Box~118, 221~00 Lund, Sweden}
\email{tomaspersson@gmx.com}

\author{Alejandro Rodriguez Sponheimer}
\address{Centre for Mathematical Sciences,
Lund University, Box~118, 221~00 Lund, Sweden}
\email{alejandro.rodriguez\_sponheimer@math.lth.se}

\date{February 6, 2025}

\subjclass[2020]{37B20 (Primary); 37D20, 37A05 (Secondary)}

\keywords{Strong Borel--Cantelli, recurrence, short return times,
multiple mixing}

\begin{document}

\begin{abstract}
  Let $(X,T,\mu,d)$ be a metric measure-preserving system for which
  $3$-fold correlations decay exponentially for Lipschitz continuous
  observables.
  Suppose that $(M_k)$ is a sequence satisfying some weak decay
  conditions and suppose there exist open balls $B_k(x)$ around $x$
  such that $\mu(B_k(x)) = M_k$.
  Under a short return time assumption, we prove a strong
  Borel--Cantelli lemma, including an error term, for
  recurrence, i.e., for $\mu$-a.e.\ $x \in X$,
  \[
    \sum_{k=1}^{n} \charfun_{B_k(x)} (T^k x) = \Phi(n) + \bigo
    \bigl( \Phi(n)^{1/2} (\log \Phi(n))^{3/2 + \varepsilon}
    \bigr),
  \]
  where $\Phi(n) = \sum_{k=1}^{n} \mu(B_k(x))$.

  Applications to systems include some non-linear piecewise expanding
  interval maps and hyperbolic automorphisms of $\torus^2$.
\end{abstract}

\maketitle

\section{Introduction}

A fundamental result in dynamics is the Poincar\'e Recurrence Theorem.
One formulation of it is, for a metric measure-preserving system
(m.m.p.s.)\ $(X,T,\mu,d)$, where $(X,d)$ is separable and $\mu$ is a
finite Borel measure, one has that $\mu$-a.e.\ point $x \in X$ is
\emph{recurrent}, i.e., for $\mu$-a.e.\ $x \in X$ there exists a
sequence $n_i \to \infty$ such that
\[
  \lim_{i \to \infty} d(x,T^{n_i}x) = 0.
\]
A shortcoming of the theorem is that it is only qualitative and not
quantitative; it says nothing about the rate of convergence. A
pioneering result establishing such a rate is due to Boshernitzan
\cite{Boshernitzan} in \textup{1993} and assumes very little of the
underlying space.

\begin{theorem}[Boshernitzan]
  Let $(X,T,\mu,d)$ be an m.m.p.s.\ and assume that, for some
  $\alpha>0$, the Hausdorff $\alpha$-measure, $\haus_\alpha$, is
  $\sigma$-finite on $X$. Then for $\mu$-a.e.\ $x \in X$,
  \[
    \liminf_{n \to \infty} n^{1/\alpha} d(x,T^{n}x) < \infty.
  \]
  Moreover, if $\haus_\alpha(X) = 0$, then for $\mu$-a.e.\ $x\in X$,
  \[
    \liminf_{n \to \infty} n^{1/\alpha} d(x,T^{n}x) = 0.
  \]
\end{theorem}

One may restate Boshernitzan's result: given the assumptions, for
$\mu$-a.e.\ ~$x$, there exists a constant $c(x) > 0$ such that
\[
  \sum_{k=1}^{\infty} \charfun_{B(x,c(x)k^{-1/\alpha})}(T^{k}x)
  = \infty,
\]
where $B(x,r)$ denotes the open ball around $x$ with radius $r$. More
generally, one can consider the series
\[
  \sum_{k=1}^{\infty} \charfun_{B(x,r_k(x))}(T^{k}x)
\]
for various decaying radii $r_k(x)$.
Results establishing the $\mu$-a.e.\ convergence and divergence of the
series depending on the convergence and divergence of
$\sum_{k=1}^{\infty} \mu\bigl( B(x,r_k(x)) \bigr)$ are called dynamical
Borel--Cantelli lemmas for recurrence (RBC) and have been proven for
various systems with exponential mixing.
For example, Hussain, Li, Simmons and
Wang \cite{HLSW} prove an RBC with an exact dichotomy for conformal
systems in which the measure is Ahlfors regular, and Kirsebom, Kunde
and Persson \cite{KKP} prove a dynamical RBC under fewer restrictions
on the system while also proving a dichotomy result for hyperbolic
linear maps on $\torus^{d}$.
These recurrence results mirror dynamical Borel--Cantelli lemmas for
the shrinking target problem, in which the `targets',
$B_k = B(y_k, r_k)$, do not depend on the initial point $x$. Some
results of this type are \cite{Philipp, CK, Dolgopyat, Kim}.
A Borel--Cantelli lemma combining recurrence and the shrinking target
problem was established by Kleinbock and Zheng \cite{KZ} for uniformly
mixing and conformal systems.

In the case that 
$\sum_{k=1}^{\infty} \mu\bigl(B(x,r_k(x))\bigr) = \infty$,
it is sometimes possible to prove a stronger result than a $0$--$1$
law and obtain the asymptotic \emph{rate} of recurrence
\begin{equation}
  \label{eq:SBC}
  \sum_{k=1}^{n} \charfun_{B(x,r_k(x))}(T^{k}x)
  \sim \sum_{k=1}^{n} \mu\bigl(B(x,r_k(x))\bigr).
\end{equation}
We will call results establishing such a limit `strong Borel--Cantelli
lemmas for recurrence' (RSBC).
Such results have been studied only recently in
\cite{He, LLSV, Persson, ARS}.
In \cite{Persson}, Persson established an RSBC lemma for interval maps
satisfying exponential decay of correlations. Persson assumes that the
measures of the balls $\mu\bigl(B(x,r_k(x))\bigr)$ do not decay too
quickly and decay in a sufficiently well-behaved manner. Although these
assumptions are weak enough to include balls decaying at a rate slower
than $1/k^{\gamma}$ for any $\gamma \in (0,1)$, they do exclude the
critical decay rate of $1/k$. Persson's result was subsequently
extended by Rodriguez Sponheimer \cite{ARS} to a class of higher
dimensional systems including Axiom~A diffeomorphisms. Rodriguez
Sponheimer assumes $3$-fold decay of correlations, thin annuli and the
same decay conditions on $\mu\bigl( B(x,r_k(x)) \bigr)$ as Persson.
More recently, He \cite{He} proved an RSBC lemma for expanding maps
on $[0,1]^{d}$ preserving an absolutely continuous measure and for
which $B_k(x)$ are a sequence of hyperrectangles with sides
$r_{k,1},\dotsc,r_{k,d}$. The restrictions allow for He to relax the
decay assumption of the previous two papers, requiring only that
$\lim_{k \to \infty} \max_{n=1,\dotsc,d}r_{k,n} = 0$.
Although his main result does not include an error term, earlier on in
his paper, He proves a different statement that does include an error
term.
In \cite{LLSV}, Levesley, Li, Simmons and Velani prove an RSBC lemma
with an error term for piecewise linear expanding maps preserving the
Lebesgue measure while completely removing any decay condition on
$r_{k,n}$ apart from
$\sum_{k=1}^{\infty} r_{k,1} \ldots r_{k,d} = \infty$.

Much like in the case of dynamical Borel--Cantelli lemmas, the results
of the previous paragraph mimic those of dynamical strong
Borel--Cantelli lemmas for the shrinking targets problem. Many papers
establishing dynamical Borel--Cantelli lemmas also establish dynamical
strong Borel--Cantelli lemmas with an error term (see the
aforementioned \cite{Philipp, CK, Kim}).

In this paper, we improve existing results \cite{Persson,ARS} under the
additional assumption of a short return time condition. We establish an
error term for the limit in \eqref{eq:SBC} and are able to replace the
decay conditions on $\mu\bigl(B(x,r_k(x))\bigr)$ with one that includes
the critical rate $1/k$. We also improve correlation estimates found in
\cite{Persson, ARS}.
Under the aforementioned assumptions, we can enlarge the class of
measures considered by He \cite{He} to measures that are not
necessarily absolutely continuous. We also consider systems for which
correlations decay exponentially for Lipschitz continuous observables
instead of H{\"o}lder continuous versus $L^{1}$ observables.
We can also extend the result of \cite{LLSV} to measures other than the
Lebesgue measure, again, under the aforementioned assumptions.

In order to state the main results, we require some setting up;
instead, we momentarily delay them and settle on stating two main
applications for the time being.
In the following two propositions, we consider a monotonically
decreasing sequence $(M_k)$ such that and $M_k \leq C (\log k)^{-2}$
for some constant $C > 0$. For each $k$, we denote by $B(x,r_k(x))$ the
open ball around $x$ with measure $M_k$, which exists under the
assumptions of the propositions. Finally, let
$\Phi(n) = \sum_{k=1}^{n} \mu\bigl(B(x,r_k(x))\bigr)
= \sum_{k=1}^{n} M_k$.

\begin{proposition}
  \label{prop:interval}
  Suppose $T \colon [0,1] \to [0,1]$ is a piecewise expanding map and
  that the intervals $C_j, j=1,\dotsc,n$ are a finite partition of
  $[0,1]$ such that $(0,1) \subset T(C_j)$ for each $j$. Assume further
  that $T|_{C_j}$ is smooth for each $j$ and that $3 \leq |T'| \leq D$
  for some constant $D \geq 3$.
  Let $\mu$ be a Gibbs measure to a H{\"o}lder continuous potential
  with zero pressure.
  Then for all $\varepsilon > 0$,
  \[
    \sum_{k=1}^{n} \charfun_{B(x,r_k(x))}(T^{k}x)
    = \Phi(n)+ \bigo \bigl( \Phi(n)^{1/2} (\log
    \Phi(n))^{3/2+\varepsilon} \bigr)
  \]
  for $\mu$-a.e.\ $x$.
\end{proposition}

Since in the following theorem, $X = \torus^2$ and $\mu = \leb$ is the
Lebesgue measure, we have that $r_k(x) = r_k$ and $M_k = \pi r_k^2$.
Thus $\Phi(n) = \pi \sum_{k=1}^{n} r_k^2$.

\begin{proposition}
  \label{prop:torus}
  Let $A$ be a $2 \times 2$ integer matrix with
  $\lvert\det A\rvert = 1$ and no eigenvalues on the unit circle.
  Define $T \colon \torus^2 \to \torus^2$ by $Tx = Ax \mod 1$.
  Then for all $\varepsilon > 0$,
  \[
    \sum_{k=1}^{n} \charfun_{B(x,r_k)}(T^{k}x)
    = \Phi(n)+ \bigo \bigl( \Phi(n)^{1/2} (\log
    \Phi(n))^{3/2+\varepsilon} \bigr)
  \]
  for $\leb$-a.e.\ $x$.
\end{proposition}

\begin{remark}
We make a number of remarks on the above propositions.
\begin{enumerate}
  \item Note that each proposition implies that if
    $\sum_{k=1}^\infty M_k < \infty$, then for almost every
    $x$, there are only finitely many $n$ for which $T^n (x) \in
    B(x,r_k(x))$. The propositions therefore provide a dichotomy of
    finitely--infinitely many returns, under the extra assumption
    that $M_k$ is decreasing and that $M_k \leq C (\log k)^2$.
  \item Our application to higher dimensional systems is limited since
    the proof of our main result, Theorem~\ref{thm:dynamical}, makes
    use of a short return time condition that is difficult to verify
    for higher dimensional non-linear systems.
  \item The conclusion of Proposition~\ref{prop:interval} holds true
    for the Gauss map even though the collection of intervals is
    infinite. We exclude the proof for sake of simplicity.
  \item We would like to contrast the condition
    $M_k \leq C (\log k)^{-2}$ in Proposition~\ref{prop:interval} with
    the condition
    $\limsup_{n \to \infty} \sum_{k=c\log n}^{n} M_k = \infty$
    in Theorem~B \cite{KKP} in which the authors prove a RBC lemma.
  \item Note that Theorem~D \cite{KKP} establishes an RBC with an exact
    $0$--$1$ dichotomy for hyperbolic toral automorphisms.
    Proposition~\ref{prop:torus} strengthens the conclusion when
    $\sum_{k=1}^{\infty}M_k = \infty$ under the assumption that
    $M_k \leq C (\log k)^{-2}$. One expects that the stronger
    conclusion holds without this additional condition on $(M_k)$.
\end{enumerate}
\end{remark}

We conclude the introduction by mentioning an application of strong
Borel--Cantelli lemmas. In 2007, Galatolo and Kim \cite{GK} established
a relationship between hitting times and the pointwise dimension of the
measure for systems satisfying a strong Borel--Cantelli lemma for the
shrinking target problem.
The proof can, with straight-forward modifications, be applied to
systems for which an RSBC lemma holds (see the corollary in
\cite{Persson} for a simplified proof). In this case, one instead
establishes a relationship between return times and the pointwise
dimension of the measure. In combination with the Theorem~1 of Barreira
and Saussol \cite{BS}, it directly implies an equality between the
lower (respectively upper) recurrence rate and the lower (respectively
upper) pointwise dimension at a point $x$ (see \cite[\S2.3]{ARS}).

\subsection{Paper structure}

We begin by describing the setting and listing the assumptions on the
systems under consideration in Section~\ref{sec:results}.
After which, we list the main results, namely Theorem~\ref{thm:general}
and Theorem~\ref{thm:dynamical}. Theorem~\ref{thm:general} is a general
result for arbitrary sets $E_k$ satisfying several measure estimates.
Theorem~\ref{thm:dynamical}, in which we restrict to
$E_k = \{x \in X : T^{k}x \in B(x,r_k(x))\}$, is a result for the
dynamical systems under consideration.
We then state some lemmas that aid in proving the main results,
including an improvement to existing correlation estimates
(Lemma~\ref{lem:largel}).
The main results are proven in Section~\ref{sec:mainproofs}.
In Section~\ref{sec:largelproof}, we prove Lemma~\ref{lem:largel} using
a particular partition of the space, which is described in
Lemma~\ref{lem:partition}.
In Section~\ref{sec:applications}, we apply Theorem~\ref{thm:dynamical}
and prove Proposition~\ref{prop:interval} and
Proposition~\ref{prop:torus}. In both cases we establish that the set
of points with short return times has small measure.

\section{Results}
\label{sec:results}

\subsection{Setting and notation}

From now on, we assume that $(X,T,\mu,d)$ is a metric
measure-preserving system, i.e., $(X,d)$ is a metric space,
$T \colon X \to X$ is a measurable transformation and $\mu$ is a Borel
probability measure for which $\mu(T^{-1} A) = \mu(A)$ for all
$\mu$-measurable sets $A$. We restrict $X$ to be a compact Riemannian
manifold (although this can be relaxed to metric spaces with
appropriate assumptions) and denote the Lebesgue measure on $X$ by
$\leb$.
Let $(M_k)$ denote a sequence in $[0,1]$ and define -- when
possible -- the open balls $B_k(x) = B(x,r_k(x))$ around $x$ by
$\mu(B_k(x)) = M_k$.
Let $\overline{P}$, $\# P$ and $|P|$ denote the closure, cardinality
and diameter, respectively, of a set $P$. For a collection of sets
$\pcal$, we define its diameter as $|\pcal| = \sup_{P \in \pcal} |P|$.
For a Lipschitz continuous function $f \colon X \to Y$ between metric
spaces, we define
\[
  [f]_1 = \sup_{x \neq y} \frac{d(f(x),f(y))}{d(x,y)}
\]
and
\[
  \lipnorm{f} = \supnorm{f} + [f]_1.
\]
For a bounded function $f \colon [0,1] \to \reals$, we write
\[
  \bvnorm{f} = \supnorm{f} + \var f,
\]
where $\var f$ is the total variation of $f$, and denote the space of
functions with $\bvnorm{f} < \infty$ by $BV$.

We will occasionally write $a \geqs b$ for expressions $a$ and $b$ if
there exists a constant $c > 0$, depending only on $(X,T,\mu,d)$, such
that $a \geq c b$.
Lastly, for sake of notational ease, $C$ will denote a multiplicative
constant that depends only on $(X,T,\mu,d)$ and may denote different
constants from equation to equation.

Note that it is not possible to define the balls $B_k(x)$ and radii
$r_k(x)$ for all sequences $(M_k)$. We will soon restrict our systems
and sequences enough so that it is possible to define $B_k(x)$ for
$\mu$-a.e.\ $x$. When it is possible, we define
\begin{equation}
  \label{eq:defn:radius}
  r_k(x) = \inf\{r \geq 0 : \mu(B(x,r)) \geq M_k\}.
\end{equation}

We now state our assumptions about the underlying systems that we 
consider, the first of which being a multiple mixing property.

\begin{assumption}[$3$-fold decorrelation]
  \label{assumption:doc}
  There exists constants $C, \tau > 0$ and $\sigma \in (0,1]$ such that
  for all non-negative Lipschitz continuous functions
  $\varphi_1, \varphi_2, \varphi_3 \colon X \to [0,\infty)$,
  \begin{multline}
    \label{eq:doc:total}
    \biggl| \int_{X} \varphi_1 \varphi_2 \circ T^{k} \varphi_3 \circ
    T^{k+l} \diff\mu - \int_{X} \varphi_1 \diff\mu \int_{X} \varphi_2
    \diff\mu \int_{X} \varphi_3 \diff\mu \biggr| \\
    \leq C \lipnorm{\varphi_1} \lipnorm{\varphi_2} \lipnorm{\varphi_3}
    e^{-\tau \min \{k,l\}}
  \end{multline}
  for all integers $k,l > 0$, and
  \begin{multline}
    \label{eq:doc:partial}
    \biggl| \int_{X} \varphi_1 \varphi_2 \circ T^{k} \varphi_3 \circ T^{k+l}
    \diff\mu - \int_{X} \varphi_1 \varphi_2 \circ T^{k} \diff\mu
    \int_{X} \varphi_3 \diff\mu \biggr| \\
    \leq C \lipnorm{\varphi_1} \lipnorm{\varphi_2} \lipnorm{\varphi_3}
    e^{-\tau l}
  \end{multline}
  for all $k \leq \sigma l$.
\end{assumption}

\begin{assumption}[Short return times]
  \label{assumption:srt}
  Let $E_k = \{x \in X : T^{k}x \in B(x,r_k(x))\}$. There exists
  constants $C, \delta, N > 0$ and a summable function $\psi \colon
  \naturals \to [0,\infty)$ such that
  \begin{equation}
    \mu (E_k \cap E_{k+l}) \leq C \mu (E_k)^{1+\delta} (\log k)^{N}
    + \mu(E_k) \psi(l).
  \end{equation}
\end{assumption}

\begin{assumption}[Frostman property]
  \label{assumption:frostman}
  There exists $C,s > 0$ such that for all $x \in X$ and $r$
  sufficiently small,
  \[
    \mu\bigl( B(x, r) \bigr) \leq C r^{s}.
  \]
\end{assumption}

\begin{assumption}[Thin annuli]
  \label{assumption:thinAnnuli}
  There exists $C,\alpha_0,r_0 > 0$ such that for all
  $\varepsilon \leq r \leq r_0$ and all $x\in X$,
  \[
    \mu\bigl( B(x, r + \varepsilon) \setminus B(x, r) \bigr)
    \leq C\varepsilon^{\alpha_0}.
  \]
\end{assumption}

\begin{remark}[Multiple mixing]
  Assumption~\ref{assumption:doc} is a weaker form of a multiple mixing
  property as stated, for example, in \cite[Theorem~7.41]{CM}, which
  holds for uniformly hyperbolic billiards.
  Kotani and Sunada \cite{KS} proved that \eqref{eq:doc:total} holds
  for topologically transitive Anosov diffeomorphisms, and it can be
  proven (see Section~\ref{sec:anosov}) in this case that
  \eqref{eq:doc:total} implies \eqref{eq:doc:partial}.
  In particular, Assumption~\ref{assumption:doc} holds for the maps
  considered in Proposition~\ref{prop:torus}.
  A similar property to \eqref{eq:doc:total} is proven by Dolgopyat
  \cite[Theorem~2]{Dolgopyat} for maps that are `strongly u-transitive
  with exponential rate.' Dolgopyat gives several examples of partially
  hyperbolic systems that satisfy such a property.
  See also property $(\pcal_r)$ in P\`ene \cite{Pene} for yet another
  similar property.

  We also point out that a stronger property than
  Assumption~\ref{assumption:doc} holds in the setting of
  Proposition~\ref{prop:interval}, that is, for some interval maps
  preserving a Gibbs measure.
  In that setting, it is established in \cite{LSV} that correlations
  decay exponentially for $L^{1}$ versus $BV$ observables:
  \[
    \biggl| \int_{X} \varphi_1 \varphi_2 \circ T^{k} \diff\mu
    - \int_{X} \varphi_1 \diff\mu \int_{X} \varphi_2 \diff\mu \biggr|
    \leq C \bvnorm{\varphi_1} \lnorm{\varphi_2}{1} e^{-\tau n}.
  \]
  This implies \eqref{eq:doc:total} for $\varphi_1, \varphi_2 \in BV$ 
  and $\varphi_3 \in L^{1}$, which implies
  Assumption~\ref{assumption:doc} once we restrict to Lipschitz
  continuous observables (see Section~\ref{sec:intervalmaps}).
\end{remark}

\begin{remark}
  Assumption~\ref{assumption:srt} can be interpreted as a short return
  time condition. It is difficult to verify for higher dimensional
  hyperbolic systems, although it appears to be a reasonable assumption
  for sufficiently chaotic systems. As we will see, this assumption
  holds for linear Anosov maps on the Torus with Lebesgue measure, and
  non-linear interval maps with Gibbs measure.

  Assumption~\ref{assumption:frostman} is a standard assumption and
  a large class of measures satisfy it.

  Assumption~\ref{assumption:thinAnnuli} is a thin annuli assumption
  and holds for various systems.
  For example, it employed in \cite{GHN} and holds for a broad class of
  Lozi maps.
  It is also used in \cite{HNPV} and holds for dispersing billiard
  systems, compact group extensions of Anosov systems, a class of Lozi
  maps and one-dimensional non-uniformly expanding interval maps
  preserving an absolutely continuous measure with density in
  $L^{p}$ for some $p > 1$.
  In \cite{ARS} it is proved that if $X$ is a smooth Riemannian
  manifold and Assumption~\ref{assumption:frostman} is satisfied for
  $s > \dim X - 1$, then Assumption~\ref{assumption:thinAnnuli} holds.
  In particular, Assumption~\ref{assumption:thinAnnuli} holds when
  $\diff\mu = h \diff\leb$ with $h \in L^{p}$ for some $p > 1$.
\end{remark}

By \cite[Lemma~3.1]{ARS}, the compactness of $X$ and
Assumptions~\ref{assumption:frostman} and \ref{assumption:thinAnnuli}
imply that $r_k$ in \eqref{eq:defn:radius} is well defined on
$\supp \mu$ (and hence a.e.)\ and is Lipschitz continuous with
$[r_k]_1 = 1$.

\subsection{Main results}

We begin by stating a general theorem for an arbitrary sequence of sets
$(E_k)$ satisfying several measure estimates.

\begin{theorem}[General Theorem]
  \label{thm:general}
  Suppose that the sets $E_k$ satisfy the following measure estimates.
  There exists constants $C,N,\delta,\tau > 0$ and $\sigma \in (0,1]$
  such that,
  \begin{align}
    \label{eq:nottoobig}
    \mu (E_k) & \leq C/(\log k)^{(2+N)/\delta} & k > 0, \\
    \label{eq:mixing}
    \mu (E_k \cap E_{k+l}) & \leq \mu (E_k) \mu (E_{k+l}) + C
                             (e^{-\tau k} + e^{- \tau l})
                             & k,l > 0, \\
    \label{eq:quasiindependence}
    \mu (E_k \cap E_{k+l}) &\leq C \mu (E_k)^{1+\delta} (\log k)^{N}
    + \mu(E_k) \Psi(l) & l \leq (\log k)^2, \\
    \label{eq:largel}
    \mu(E_k \cap E_{k+l}) &\leq \mu(E_k) \mu(E_{k+l})
    + C\mu(E_{k+l})e^{-\tau k} + Ce^{-\tau l} & k \leq \sigma l.
  \end{align}
  Then
  \[
    \sum_{k=1}^{n} \charfun_{E_k} = \Phi(n) + \bigo \bigl(
    \Phi(n)^{1/2} (\log \Phi(n))^{3/2 + \varepsilon} \bigr)
  \]
  $\mu$-almost everywhere, where $\Phi(n) = \sum_{k=1}^{n} \mu(E_k)$.
\end{theorem}

We will use Theorem~\ref{thm:general} with
$E_k = \{x \in X : T^{k}x \in B(x,r_k(x))\}$ to prove the following
result. Recall that we assume that $X$ is a compact Riemannian
manifold.

\begin{theorem}
  \label{thm:dynamical}
  Suppose that $(X,T,\mu,d)$ is an m.m.p.s.\ satisfying
  Assumptions~\ref{assumption:doc}--\ref{assumption:thinAnnuli} and
  that $(M_k)$ is a sequence such that
  $M_k \leq C(\log k)^{-(2+N)/\delta}$ for some $C > 0$ and the
  $N,\delta > 0$ appearing in Assumption~\ref{assumption:srt}. Then
  \[
    \sum_{k=1}^{n} \charfun_{B(x,r_k(x))} (T^k x) = \Phi(n) +
    \bigo \bigl( \Phi(n)^{1/2} (\log \Phi(n))^{3/2 + \varepsilon}
    \bigr)
  \]
  for $\mu$-a.e.\ $x \in X$, where
  $\Phi(n) = \sum_{k=1}^{n} \mu\bigl(B(x,r_k(x))\bigr)$.
\end{theorem}

\subsection{Lemmas}

A central result that we use to establish Theorem~\ref{thm:general} is
the following result originally due to G\'al and Koksma
\cite{GalKoksma} and subsequently generalised in \cite{Harman} and
\cite[Lemma~10]{VGS}).

\begin{lemma}[G\'al--Koksma]
  \label{lem:galkoksma}
  Suppose that $\mu(E_k) \leq \phi_k$ and that
  \[
    \int \biggl( \sum_{k=m}^n \mathbbm{1}_{E_k} - \mu (E_k)
    \biggr)^2 \diff\mu \leq C \sum_{k=m}^n \phi_k.
  \]
  Let
  \[
    \Phi (n) = \sum_{k=1}^n \phi_k.
  \]
  Then for all $\varepsilon > 0$,
  \[
    \sum_{k=1}^n \mathbbm{1}_{E_k} = \sum_{k=1}^n \mu (E_k) + \bigo
    (\Phi (n)^{1/2} (\log \Phi (n))^{3/2 + \varepsilon}),
  \]
  $\mu$-almost everywhere.
\end{lemma}

The following lemma was established in \cite[Proposition~4.1]{ARS} and
extends the estimates in \cite[Lemma~4.1, Lemma~4.2]{KKP} to
higher dimensional systems but with the additional assumption that
$M_k \to 0$. The estimates in \cite{KKP} hold for interval maps with
exponential decay of correlations for $L^{1}$ versus $BV$ observables
and measures satisfying Assumption~\ref{assumption:frostman} without
any restrictions on the sequence $(M_k)$.
\begin{lemma}
  \label{lem:estimates:general}
  Let $E_k = \{x \in X : T^{k}x \in B(x,r_k(x))\}$ and suppose that
  Assumptions~\ref{assumption:doc}, \ref{assumption:frostman} and
  \ref{assumption:thinAnnuli} hold.
  Suppose further that $(M_k)$ converges to $0$. Then there exists
  $C, \eta_0 > 0$ such that for all $k,l \in \naturals$,
  \begin{equation}
    \label{eq:Ekmeasure:general}
    | \mu(E_k) - M_k | \leq C e^{-\eta_0 k}
  \end{equation}
  and
  \begin{equation}
    \label{eq:correlation:old}
    \mu(E_k \cap E_{k+l}) \leq \mu(E_k) \mu(E_{k+l}) + Ce^{-\eta_0 k}
    + Ce^{-\eta_0 l}.
  \end{equation}
\end{lemma}

In the proof of Theorem~\ref{thm:dynamical}, we require a better
estimate than \eqref{eq:correlation:old} when $k \leq \sigma l$ for
some $\sigma \in (0,1]$. This is the contents of the following lemma,
which will be proven in Section~\ref{sec:largelproof}.

\begin{lemma}
  \label{lem:largel}
  Let $E_k = \{x \in X : T^{k}x \in B(x,r_k(x))\}$ and suppose that
  Assumptions~\ref{assumption:doc}, \ref{assumption:frostman} and
  \ref{assumption:thinAnnuli} hold. Let $\sigma \in (0,1]$ be as in
  Assumption~\ref{assumption:doc} and suppose further that $(M_k)$ 
  converges to $0$. Then there exists $C, \eta_1 > 0$ such that for all
  $k \leq \sigma l$,
  \[
    \mu(E_k \cap E_{k+l}) \leq \mu(E_k) \mu(E_{k+l})
    + C \mu(E_{k+l}) e^{-\eta_1 k} + e^{-\eta_1 l}.
  \]
\end{lemma}

When $k$ is large and $l$ is small, we find that the estimates in the
previous two lemmas are insufficient for our proof of
Theorem~\ref{thm:dynamical}, which is why we impose
Assumption~\ref{assumption:srt}. Thus, to be able to apply
Theorem~\ref{thm:dynamical} to various systems, we must verify
Assumption~\ref{assumption:srt}. This will be done for specific
systems in Section~\ref{sec:applications}.

\section{Proof of Main Theorems}
\label{sec:mainproofs}

We prove Theorem~\ref{thm:general} first and then use it to prove
Theorem~\ref{thm:dynamical}.

\subsection{Proof of the general theorem}

\begin{proof}[Proof of Theorem~\ref{thm:general}]
  We start by noticing that if
  $\sum_{k=1}^\infty \mu(E_k) < \infty$, then by the easy part of
  the Borel--Cantelli lemma, $\mu (\limsup E_n) = 0$ and as a
  consequence, for almost all $x$, the sum
  $\sum_{k=1}^n \mathbbm{1}_{E_k} (x)$ is bounded from above as
  $n$ grows. Since $\Phi (n)$ is also bounded in this case, there
  is nothing to prove.
  
  From now on, we assume that $\sum_{k=1}^\infty \mu
  (E_k)$. Without loss of generality, we may assume that
  $\mu (E_k) \geq C k e^{- \tau (\log k)^2}$. Indeed, we may
  replace any instance of $\mu(E_k)$ that does not satisfy the
  inequality by $C k e^{- \tau (\log k)^2}$. The new sum
  $\Phi(n)$ differs from the old sum only by a constant, which
  can be included into the error term.

  We take $1 \leq m \leq n$ and consider the sum
  \[
    S_{m,n} = \int \biggl( \sum_{k = m}^n \mathbbm{1}_{E_k} - \mu
    (E_k) \biggr)^2 \diff\mu.
  \]
  We may rewrite this sum as
  \[
    S_{m,n} = \sum_{k=m}^n (\mu (E_k) - \mu (E_k)^2) + 2 \sum_{k =
      m}^{n-1} \sum_{l = 1}^{n-k} (\mu (E_k \cap E_{k+l}) -
    \mu(E_k) \mu (E_{k+l})).
  \]
  Clearly, we have
  \[
    \sum_{k=m}^n (\mu (E_k) - \mu (E_k)^2) \leq \sum_{k=m}^n \mu
    (E_k),
  \]
  and we let
  \[
    \Sigma_{m,n} = \sum_{k = m}^{n-1} \sum_{l = 1}^{n-k} (\mu (E_k
    \cap E_{k+l}) - \mu(E_k) \mu (E_{k+l})).
  \]

  To estimate $\Sigma_{m,n}$ we split it into several sums
  \[
  \begin{split}
    \Sigma_{m,n} = &\sum_{k=m}^{n-1} \sum_{l=1}^{(\log k)^2}
      (\mu (E_k \cap E_{k+l}) - \mu(E_k) \mu (E_{k+l})) \\
    &+ \sum_{k=m}^{n-1} \sum_{l=(\log k)^2+1}^{k/\sigma}
      (\mu (E_k \cap E_{k+l}) - \mu(E_k) \mu (E_{k+l})) \\
    &+ \sum_{k=m}^{n-1} \sum_{l=k/\sigma+1}^{n-k}
      (\mu (E_k \cap E_{k+l}) - \mu(E_k) \mu (E_{k+l}))
  \end{split}
  \]
  and use different estimates on the terms.

  We start with terms with $l \leq (\log k)^2$ and use
  \eqref{eq:nottoobig} and \eqref{eq:quasiindependence} to get
  \begin{align*}
    \sum_{k=m}^{n-1} \sum_{l = 1}^{(\log k)^2}
    & (\mu (E_k \cap E_{k+l}) - \mu(E_k) \mu (E_{k+l})) \\
    & \leq \sum_{k=m}^n C (\mu (E_k)^{1+\delta} (\log k)^{2+N}
      + \mu (E_k) ) \\
    & \leq C \sum_{k=m}^n \mu (E_k).
  \end{align*}

  Continuing with terms with $(\log k)^2 \leq l \leq k/\sigma$, we use
  \eqref{eq:mixing} to get that
  \[
  \begin{split}
    \sum_{k=m}^{n-1} \sum_{l=(\log k)^2}^{k/\sigma} \mu(E_k \cap
      E_{k+l}) - \mu(E_k) \mu(E_{k+l})
    &\leq \sum_{k=m}^{n-1} \sum_{l=(\log k)^2}^{k/\sigma} C(e^{-\tau k}
      + e^{-\tau l}) \\
    &\leq \sum_{k=m}^{n-1} C(ke^{-\tau k} + e^{-\tau (\log k)^2}) \\
    &\leq C \sum_{k=m}^{n-1} \mu(E_k),
  \end{split}
  \]
  where we have used $\mu(E_k) \geq C k e^{-\tau (\log k)^2}$ in the
  last inequality.

  Finally, for terms with $l \geq k/\sigma$, we use \eqref{eq:largel}
  to obtain that
  \[
  \begin{split}
    \sum_{k=m}^{n-1} \sum_{l=k/\sigma}^{n-k} \mu(E_k \cap
    E_{k+l}) - \mu(E_k &) \mu(E_{k+l}) \\
    &\leq \sum_{k=m}^{n-1} \sum_{l=k/\sigma}^{n-k} \bigl( C\mu(E_{k+l})
      e^{-\tau k} + Ce^{-\tau l} \bigr) \\
    &\leq \sum_{k=m}^{n-1} \sum_{j=2k}^{n} C\mu(E_j) e^{-\tau k}
      + \sum_{k=m}^{n-1} Ce^{-\tau k} \\
    &\leq \sum_{k=m}^{n-1} \biggl( \sum_{j=2m}^{n} C\mu(E_j) \biggr)
    e^{-\tau k} + \sum_{k=m}^{n} C\mu(E_k) \\
    &\leq C \sum_{k=m}^{n} \mu(E_k),
  \end{split}
  \]
  where we have used $k/\sigma \geq k$ in the second inequality and
  $\mu(E_k) \geq C k e^{-\tau k}$ in the third inequality.

  Thus, $\Sigma_{m,n} \leq C \sum_{k=m}^{n} \mu(E_k)$ as desired and
  \[
    S_{m,n} \leq C \sum_{k=m}^n \mu (E_k).
  \]
  Theorem~\ref{thm:dynamical} now follows by the G\'{a}l--Koksma Lemma
  (Lemma~\ref{lem:galkoksma}).
\end{proof}

\subsection{Proof of the dynamical theorem}

Recall that we assume that
Assumptions~\ref{assumption:doc}--\ref{assumption:thinAnnuli} hold. We
also let $E_k = \{x \in X : T^{k}x \in B(x,r_k(x))\}$ and
$\Phi(n) = \sum_{k=1}^{n} \mu\bigl( B(x,r_k(x)) \bigr) =
\sum_{k=1}^{\infty}M_k$.

\begin{proof}[Proof of Theorem~\ref{thm:dynamical}]
  Similar to the proof of Theorem~\ref{thm:general}, we may assume,
  without loss of generality, that $M_k \geq Cke^{-\tau (\log k)^2}$
  for some $C > 0$.
  In order to apply Theorem~\ref{thm:general}, we first verify that
  \eqref{eq:nottoobig}--\eqref{eq:largel} hold.

  By \eqref{eq:Ekmeasure:general} in Lemma~\ref{lem:estimates:general}
  and the assumption on $(M_k)$, we have that
  \[
    \mu(E_k) \leq M_k + Ce^{-\tau k} \leq C M_k
    \leq \frac{C}{(\log k)^{(2 + N)/\delta}}.
  \]
  Hence, \eqref{eq:nottoobig} holds.
  The estimates \eqref{eq:mixing} and \eqref{eq:largel} hold as a direct
  consequence of Lemma~\ref{lem:estimates:general} and
  \ref{lem:largel}, respectively.
  Finally, \eqref{eq:quasiindependence} is exactly the condition in
  Assumption~\ref{assumption:srt}.
  Thus, an application of Theorem~\ref{thm:general} allows us to
  conclude that
  \begin{equation}\label{eq:Ekestimate}
    \sum_{k=1}^{n} \charfun_{E_k} = \sum_{k=1}^{n} \mu(E_k) + \bigo
    \bigl( \Psi(n)^{1/2} (\log \Psi(n))^{3/2 + \varepsilon} \bigr)
  \end{equation}
  $\mu$-a.e., where $\Psi(n) = \sum_{k=1}^{n} \mu(E_k)$. Now, by the
  estimate \eqref{eq:Ekmeasure:general} in
  Lemma~\ref{lem:estimates:general}, we obtain that
  $\Psi(n) = \Phi(n) + \bigo(1)$. Furthermore,
  $\charfun_{E_k}(x) = \charfun_{B(x,r_k(x))}(x)$ by the definition of
  $E_k$. Thus, from \eqref{eq:Ekestimate} we obtain
  \[
    \sum_{k=1}^{n} \charfun_{B(x,r_k(x))} (T^k x) =
    \sum_{k=1}^{n} \mu(B(x,r_k(x))) + \bigo \bigl( \Phi(n)^{1/2}
    (\log \Phi(n))^{3/2 + \varepsilon} \bigr)
  \]
  for $\mu$-a.e., $x \in X$, as desired.
\end{proof}

\section{Proof of Lemma~\ref{lem:largel}}
\label{sec:largelproof}

This section is dedicated to proving Lemma~\ref{lem:largel}.
We make use of the following lemma, which was communicated to us
through private correspondence with Conze and Le Borgne. It fixes an
inaccuracy in creating a partition of a Riemannian manifold with the
required properties in \cite[Theorem~3.3]{CL}.

\begin{lemma}[Conze--Le Borgne]
  \label{lem:partition}
  Suppose that $X$ is a compact $d$-dimensional smooth manifold and let
  $\kappa > 0$. Then there exists a constant $C > 0$ such that, for all
  $n \in \naturals$, there exists a subset $Y_n \subset X$ and a
  partition $\pcal_n$ of $X \setminus Y_n$ such that
  \begin{align}
    \label{eq:partitionSize}
    \# \pcal_n &\leq e^{d\kappa n}, \\
    \label{eq:diam}
    |\pcal_n| &\leq C e^{-\kappa n}, \\
    \label{eq:carveout}
    \mu(Y_n) &\leq e^{-\kappa n}.
  \end{align}
  Moreover, for each $P \in \pcal_n$, there exists a density function
  $\rho_{n,P} \colon X \to [0,\infty)$ such that
  \begin{align}
    \label{eq:rhoSup}
    \supnorm{\rho_{P,n}} &\leq C e^{(d+1) \kappa n}, \\
    \label{eq:rhoLip}
    [\rho_{P,n}]_1 &\leq C e^{2(d+2)\kappa n}, \\
    \label{eq:charApprox}
    \lnorm{\rho_{P,n} - \mu(P)^{-1} \charfun_P}{1}
      &\leq C e^{- \kappa n}.
  \end{align}
\end{lemma}

\begin{proof}
  By considering local charts, we may prove the statement when
  $X = [0,1]^{d}$. Furthermore, we demonstrate the proof for $d=2$ as
  it captures the essential idea of the general proof while keeping the
  notation simple.

  Let $\kappa > 0$. Define $\Theta_n = \lfloor e^{\kappa n} \rfloor$
  and $\theta_n = \lfloor e^{4\kappa n} \rfloor$. Partition $[0,1]^2$
  by $(\Theta_n\theta_n)^2$ squares of equal size
  \[
    \biggl\{ I(i,j) = \biggl[\frac{i-1}{\Theta_n\theta_n},
    \frac{i}{\Theta_n\theta_n}\biggr] \times
    \biggl[\frac{j-1}{\Theta_n\theta_n},
    \frac{j}{\Theta_n\theta_n}\biggr] : i,j = 1,\dotsc,\Theta_n\theta_n
    \biggr\}.
  \]
  For each $i = 1,\dotsc,\Theta_n \theta_n$ consider
  \[
    V_i^{n} = \bigcup_{j=1}^{\Theta_n\theta_n} I(i,j).
  \]
  Since $\mu(\cup_{i=(s-1)\theta_n + 1}^{s\theta_n} V_i^{n}) \leq 1$
  for each $s=1,\dotsc,\Theta_n$, we have that there exists, for each
  $s=1,\dotsc,\Theta_n$, an $i_s \in ( (s-1)\theta_n, s\theta_n]$ such
  that $\mu(V_{i_s}^{n}) \leq \theta_n^{-1}$. We may do the same for
  each $j=1,\dotsc,\Theta_n\theta_n$ and
  \[
    H_j^{n} = \bigcup_{i=1}^{\Theta_n\theta_n} I(i,j)
  \]
  to obtain, for each $s=1,\dotsc,\Theta_n$, a
  $j_s \in ( (s-1)\theta_n, s\theta_n]$ such that
  $\mu(H_{j_s}^{n}) \leq \theta_n^{-1}$. Now define the sets
  \begin{align*}
    Q_{k,l} &= \bigcup_{i=i_{k-1}+1}^{i_k-1}
    \bigcup_{j=j_{l-1}+1}^{j_l-1} I(i,j) \\
    P_{k,l} &= \bigcup_{i=i_{k-1}+1}^{i_k}
    \bigcup_{j=j_{l-1}+1}^{j_l} I(i,j) \\
    R_{k,l} &= \bigcup_{i=i_{k-1}}^{i_k}
    \bigcup_{j=j_{l-1}}^{j_l} I(i,j).
  \end{align*}
  Note that $\{P_{k,l}\}_{k,l = 1,\dotsc,\Theta_n}$ partitions
  $[0,1]^2$ and that
  $R_{k,l} \setminus Q_{k,l} \subset V_{i_{k-1}}^{n} \cup V_{i_k}^{n}
  \cup H_{j_{l-1}}^{n} \cup H_{j_l}^{n}$ and
  $\mu( R_{k,l} \setminus Q_{k,l} ) \leq C \theta_n^{-1}$.

  Define $S_0 = \{(k,l) : \mu(P_{k,l}) < e^{-3\kappa n} \}$ and
  $Y_n = \cup_{(k,l) \in S_0} P_{k,l}$. Then
  $\pcal_n = \{P_{k,l} : (k,l) \notin S_0\}$ is a partition of
  $[0,1]^2 \setminus Y_n$ with
  $\#\pcal_n \leq \Theta_n^2 \leq e^{2\kappa n}$, proving
  \eqref{eq:partitionSize}, and
  $|P_{k,l}| \leq \Theta_n^{-1} \leq C e^{-\kappa n}$ for all
  $k,l$, proving \eqref{eq:diam}. Moreover,
  $\mu(Y_n) \leq \Theta_n^2 e^{-3\kappa n} \leq e^{-\kappa n}$, which
  proves \eqref{eq:carveout}.

  We now define the densities $\rho_{n,P} \colon X \to [0,\infty)$. For
  each $P_{k,l} \in \pcal_n$ define $\psi_{k,l} \colon X \to [0,1]$ by
  \[
    \psi_{k,l} =
    \begin{cases}
      1 & \text{on } Q_{k,l} \\
      0 & \text{on } X \setminus R_{k,l} \\
      \text{affine} & \text{on } R_{k,l} \setminus Q_{k,l}
    \end{cases}
  \]
  in such a way that $\psi_{k,l}$ have disjoint support and are
  Lipschitz continuous such that for some $C > 0$ we have
  $[\psi_{k,l}]_1 \leq C \Theta_n \theta_n \leq C e^{5\kappa n}$.
  Now, for each $P_{k,l} \in \pcal_n$ define
  $\rho_{k,l} = \mu(P_{k,l})^{-1} \psi_{k,l}$. Then
  $\supnorm{\rho_{k,l}} \leq \mu(P_{k,l})^{-1} \leq e^{3\kappa n}$,
  proving \eqref{eq:rhoSup}, and
  $[\rho_{k,l}]_1 \leq Ce^{3\kappa n} e^{5\kappa n} = C e^{8\kappa n}$,
  proving \eqref{eq:rhoLip}.

  Finally, let $\gamma_{k,l} = \charfun_{P_{k,l}} - \psi_{k,l}$ for
  $P_{k,l} \in \pcal_n$. Then
  $0 \leq |\gamma_{k,l}| \leq \charfun_{R_{k,l} \setminus Q_{k,l}}$ and
  \[
  \begin{split}
    \int_X \Bigl| \rho_{k,l}
    - \frac{\charfun_{P_{k,l}}}{\mu(P_{k,l})} \Bigr| \diff\mu
    &= \frac{1}{\mu(P_{k,l})} \int_X | \psi_{k,l}
    - \charfun_{P_{k,l}} |\diff\mu \\
    &= \frac{1}{\mu(P_{k,l})} \int_X | \gamma_{k,l} |\diff\mu \\
    &\leq \frac{1}{\mu(P_{k,l})} \mu(R_{k,l} \setminus Q_{k,l}) \\
    &\leq C\frac{1}{\mu(P_{k,l})} \theta_n^{-1} \\
    &\leq C e^{3\kappa n} e^{-4\kappa n} \\
    &= C e^{-\kappa n},
  \end{split}
  \]
  which proves \eqref{eq:charApprox}.
\end{proof}

Before proving Lemma~\ref{lem:largel}, we formulate the following
lemma for convenience.

\begin{lemma}\label{lem:localise}
  Let $X$ be a compact Riemannian manifold and let $\kappa > 0$ and
  $n \in \naturals$. Suppose that
  $h \colon X \times X \times X \to \reals$ is Lipschitz continuous.
  Then
  \begin{multline*}
    \int_X h(x, T^{k}x, T^{k+l}x)
    = \sum_{P \in \pcal_n} \mu(P) \int_X \rho_{P,n}(x) h(x_{P,n},
    T^{k}x, T^{k+l}x) \diff\mu(x) \\
    + \bigo( \lipnorm{h} e^{-\kappa n} ),
  \end{multline*}
  where $x_{P,n}$ is an arbitrary point in $P \in \pcal_n$.
\end{lemma}

\begin{proof}
  Write $\rho_P = \rho_{P,n}$ and $x_P = x_{P,n}$, and let
  \[
    H(x) = \sum_{P \in \pcal_n} h(x_P, T^{k}x, T^{k+l}x) \charfun_P(x).
  \]
  Then
  \[
    \lnorm{h(x, T^{k}x, T^{k+l}x) - H(x)}{1}
    \leq \supnorm{h}\mu(Y_n) + [h]_1 \lvert \pcal_n \rvert
  \]
  since
  \[
    \Bigl| h(x, T^{k}x, T^{k+l}x) - \sum_{P\in\pcal_n} h(x, T^{k}x,
    T^{k+l}x) \charfun_P(x) \Bigr|
    \leq \supnorm{h} \charfun_{Y_n}(x)
  \]
  and
  \[
    \bigl| h(x, T^{k}x, T^{k+l}x) - h(x_P, T^{k}x, T^{k+l}x) \bigr|
    \leq [h]_1 d(x,x_P) \leq [h]_1 |\pcal_n|
  \]
  for $x \in P$. Approximating the functions $\mu(P) ^{-1} \charfun_P$
  by $\rho_P$, we obtain that
  \[
    \Bigl\| H(x) - \sum_{P\in\pcal_n} \mu(P) h(x_P,
    T^{k}x, T^{k+l}x) \rho_P(x) \Bigr\|_{L^{1}}
    \leq 
    \lnorm{\rho_P - \mu(P)^{-1} \charfun_P}{1}.
  \]
  Using the estimates from Lemma~\ref{lem:partition}, we obtain the
  desired inequality.
\end{proof}

We now prove Lemma~\ref{lem:largel}.

\begin{proof}[Proof of Lemma~\ref{lem:largel}]
  Let $\sigma \in (0,1]$ be as in Assumption~\ref{assumption:doc} and
  let $\eta> 0$, which will later be chosen sufficiently small. Fix
  $k \leq \sigma l$. We may consider Lipschitz continuous functions
  $f, g \colon X \times X \to [0,1]$ such that there exists
  $C,\eta' > 0$ independent of $k,l$ with
  \begin{gather}
    \label{eq:f:bounds}
    0 \leq \charfun_{B(x,r_k(x))}(y) \leq f(x,y) \leq 1 \\
    \label{eq:f:lip}
    \lipnorm{f} \leq Ce^{\eta k} \leq Ce^{\eta \sigma l} \\
    \label{eq:f:fixedMeas}
    \int_X f(x,y) \diff\mu(y) = \mu(E_k) + \bigo(e^{-\eta' k})
  \end{gather}
  and
  \begin{gather}
    \label{eq:g:bounds}
    0 \leq \charfun_{B(x,r_{k+l}(x))}(y) \leq g(x,y) \leq 1 \\
    \label{eq:g:lip}
    \lipnorm{g} \leq Ce^{\eta l} \\
    \label{eq:g:fixedMeas}
    \int_X g(x,y) \diff\mu(y) = \mu(E_{k+l}) + \bigo(e^{-\eta' l})
  \end{gather}
  Indeed, let
  $Z_k = \{(x,y) \in X \times X : \charfun_{B(x,r_{k}(x))}(y) = 1\}$
  and define 
  \begin{align*}
    f(x,y) &= \max \bigl\{0, 1-\dist\bigl( (x,y), Z_k \bigr)e^{\eta k} 
      \bigr\}, \\
    g(x,y) &= \max \bigl\{0, 1-\dist\bigl( (x,y), Z_{k+l} \bigr)
      e^{\eta l} \bigr\}.
  \end{align*}
  It is clear that \eqref{eq:f:bounds}, \eqref{eq:f:lip},
  \eqref{eq:g:bounds} and \eqref{eq:g:lip} are satisfied.
  Equation~\eqref{eq:f:fixedMeas} with
  $\eta' = \min \{\eta_0,\alpha_0\eta\}$ can be established using
  Assumption~\ref{assumption:thinAnnuli}, estimate
  \eqref{eq:Ekmeasure:general} in Lemma~\ref{lem:estimates:general} and
  the fact that $f(x,y) \leq \charfun_{B(x,r_k(x) + e^{-\eta k})}(y)$.
  Equation~\eqref{eq:g:fixedMeas} can be established in a similar
  manner. We therefore have that
  \[
    \mu(E_k \cap E_{k+l}) \leq \int f(x,T^{k}x) g(x,T^{k+l}x)
    \diff\mu(x).
  \]

  We first use a partition suited to separate $g$ from the integral.
  Let $\pcal_l$ be a partition as in Lemma~\ref{lem:partition} with
  $\kappa = 3 \eta$. We assign an arbitrary point $x_{P,l} \in P$ to
  each element $P \in \pcal_l$. Applying Lemma~\ref{lem:localise} to
  the function $h(x,y,z) = f(x,y) g(x,z)$ yields
  \begin{multline*}
    \int f(x,T^{k}x) g(x,T^{k+l}x) \diff\mu(x) \\
    \leq \sum_{P\in\pcal_l} \mu(P) \int_X \rho_{P,l}(x) f(x_{P,l},T^{k}x)
    g(x_{P,l},T^{k+l}x) \diff\mu(x) + \bigo(e^{-\eta l}), 
  \end{multline*}
  where we have used \eqref{eq:f:lip} and \eqref{eq:g:lip} (and
  $\sigma \leq 1$) to obtain that
  \[
    \bigo(\lipnorm{f} \lipnorm{g} e^{-\kappa l}) = \bigo(
    e^{(1+\sigma)\eta l - 3\eta l} ) = \bigo( e^{-\eta l}).
  \]
  For each $P \in \pcal_l$, applying the decorrelation estimate
  \eqref{eq:doc:partial} yields
  \begin{multline*}
    \int \rho_{P,l}(x) f(x_{P,l},T^{k}x) g(x_{P,l},T^{k+l}x)
    \diff\mu(x) \\
    = \int \rho_{P,l}(x) f(x_{P,l},T^{k}x) \diff\mu(x) \int_X
    g(x_{P,l},x) \diff\mu(x) \\
    + \bigo\bigl( \lipnorm{\rho_{P,l}} \lipnorm{f} \lipnorm{g}
    e^{-\tau l} \bigr).
  \end{multline*}
  By Lemma~\ref{lem:partition} and the estimates \eqref{eq:f:lip} and
  \eqref{eq:g:lip}, we have that
  \[
    \lipnorm{\rho_{P,l}} \lipnorm{f} \lipnorm{g}
    = \bigo\bigl(  e^{6(d+2)\eta l} e^{\eta \sigma l} e^{\eta l}\bigr)
    = \bigo \bigl( e^{(6d+13+\sigma)\eta l} \bigr).
  \]
  Hence,
  \begin{multline*}
    \mu(E_k \cap E_{k+l})
    \leq \sum_{P\in\pcal_l} \mu(P) \int \rho_{P,l}(x) f(x_{P,l},T^{k}x)
    \diff\mu(x) \int g(x_{P,l},x) \diff\mu(x) \\
    + \bigo(e^{-\eta l}) + \bigo \bigl( e^{(6d+14)\eta l} e^{-\tau l}
    \bigr).
  \end{multline*}
  Using \eqref{eq:g:fixedMeas}, we obtain that
  \begin{multline}\label{eq:gremoved}
    \mu(E_k \cap E_{k+l}) \leq \mu(E_{k+l}) \sum_{P\in\pcal_l} \mu(P)
    \int \rho_{P,l}(x) f(x_{P,l},T^{k}x) \diff\mu(x) \\
    + \bigo(e^{-\eta l}) + \bigo(e^{-\eta' l})
    + \bigo \bigl( e^{(6d+14)\eta l} e^{-\tau l} \bigr).
  \end{multline}

  We now move on to separate $f$ from the integral, but we require a
  more suitable partition. By Lemma~\ref{lem:localise} (with
  $h(x,y,z) = f(x,y)$),
  \[
    \sum_{P\in\pcal_l} \mu(P) \int_X \rho_{P,l}(x) f(x_{P,l},T^{k}x)
    \diff\mu(x)
    = \int_X f(x,T^{k}x) \diff\mu(x) + \bigo(e^{-2\eta l}).
  \]
  Using a different partition $\pcal_k$ with $\kappa = 2 \eta$, we
  obtain from Lemma~\ref{lem:localise} and \eqref{eq:f:lip} that
  \[
  \begin{split}
    \int_X f(x,T^{k}x) \diff\mu(x)
    &= \sum_{P\in\pcal_k} \mu(P) \int_X \rho_{P,k}(x) f(x_{P,k},T^{k}x)
    \diff\mu(x) + \bigo( e^{-\eta k} ),
  \end{split}
  \]
  where $x_{P,k} \in P$ is an arbitrary point for all $P \in \pcal_k$.
  By decay of correlations,
  \begin{multline*}
    \sum_{P\in\pcal_k} \mu(P) \int_X \rho_{P,k}(x) f(x_{P,k},T^{k}x)
    \diff\mu(x) = \sum_{P\in\pcal_k} \mu(P) \int_X f(x_{P,k},x)
    \diff\mu(x) \\
    + \bigo( \lipnorm{\rho_{P,k}} \lipnorm{f} e^{-\tau k} ),
  \end{multline*}
  where we have used that
  $\int_X \rho_{P,k} \diff\mu = 1 + \bigo(e^{-\eta k})$.
  By \eqref{eq:rhoLip} and \eqref{eq:f:lip},
  \begin{multline*}
    \int_X f(x,T^{k}x) \diff\mu(x)
    = \sum_{P\in\pcal_k} \mu(P) \int_X f(x_{P,k},x) \diff\mu(x) \\
    + \bigo( e^{-\eta k} ) + \bigo( e^{(4d+9)\eta k} e^{-\tau k} ).
  \end{multline*}
  By \eqref{eq:f:fixedMeas}, we obtain 
  \[
    \int_X f(x,T^{k}x) \diff\mu(x)
    = \mu(E_k) + \bigo(e^{-\eta k}) + \bigo(e^{-\eta' k}) 
    + \bigo( e^{(4d+9)\eta k} e^{-\tau k} ).
  \]

  Combining \eqref{eq:gremoved} with the previous equation and choosing
  \[
    \eta = \frac{\tau}{2(6d+14)}
  \]
  yields
  \[
    \mu(E_k \cap E_{k+l}) \leq \mu(E_k)\mu(E_{k+l})
    + C \mu(E_{k+l}) e^{-\eta_1 k} + C e^{-\eta_1 l},
  \]
  where $\eta_1 = \min \{\eta, \eta', \tau/2\}
  = \min \{\eta, \alpha_0 \eta, \eta_0, \tau/2\}$.
\end{proof}

\section{Applications}
\label{sec:applications}

In this section we apply Theorem~\ref{thm:dynamical} to piecewise
expanding systems on the interval and to linear hyperbolic systems on
the two-dimensional torus. We find that, under the assumptions of
Proposition~\ref{prop:interval} and Proposition~\ref{prop:torus}, the
only remaining part is to prove that short returns have small measure.

Recall that for Propositions~\ref{prop:interval} and \ref{prop:torus},
$(M_k)$ is a monotonically decreasing sequence such that
$M_k \leq C (\log k)^{-2}$ for some $C > 0$.

\subsection{Interval maps preserving a Gibbs measure}
\label{sec:intervalmaps}

Here we consider the case when $T \colon [0,1] \to [0,1]$ is a
piecewise expanding map. We assume that there is a finite partition of
$[0,1]$ into intervals $C_j$, $j = 1,\ldots,n$ and that
$(0,1) \subset T (C_j)$ for each
$j$. We also assume that $T$ restricted to $C_j$ is smooth and
that $3 \leq |T'| \leq D$ for some constant $D \geq 3$.

For any $n \geq 1$, we have that $T^n$ is also piecewise
expanding with respect to a partition consisting of the cylinder
sets
\[
  C_{j_1,\ldots,j_n} = \bigcap_{m = 1}^n T^{-m+1} (C_{j_m}).
\]

We let $\mu$ be a Gibbs measure to a H{\"o}lder continuous potential
$\phi$ and we assume that the pressure of $\phi$ is zero. A consequence
of this is that there is a constant $C > 0$ such that for all $n > 0$ 
\[
  C^{-1} \leq \frac{\mu(C_{j_1,\dotsc,j_k})}{\exp(-S_n\phi(x))} \leq C,
\]
where $S_n\phi = \phi + \phi \circ T + \dotsc + \phi \circ T^{n-1}$ and
$x$ is any point in $C_{j_1,\ldots,j_n}$.

Moreover, it follows from the proof of Proposition~1.14 in \cite{Bowen}
that the measure $\mu$ satisfies exponential `\mbox{$\psi$-mixing}',
and this implies directly that $\mu$ satisfies the quasi-Bernoulli
property, i.e., there is a constant $C > 1$ such that for all $n,m>0$
\begin{equation}
  \label{eq:quasiBernoulli}
  \mu(C_{j_1,\dotsc,j_{n+m}}) \leq C \mu(C_{j_1,\dotsc,j_n})
  \mu(C_{j_{n+1},\dotsc,j_{n+m}}).
\end{equation}
Furthermore, it is proven in \cite{GS} that exponential $\psi$-mixing
implies that the measure of cylinders decay exponentially with their
length, i.e., there are constants $C > 0$ and $\lambda > 1$ such that
for all $n > 0$
\begin{equation}
  \label{eq:cylinderdecay}
  \mu(C_{j_1,\dotsc,j_{n}}) \leq C \lambda^{-n}.
\end{equation}

We see directly that Assumption~\ref{assumption:frostman} and
Assumption~\ref{assumption:thinAnnuli} are satisfied from the above
properties. Furthermore, it is established in \cite{LSV} that these
systems satisfy decay of correlation for $L^{1}$ versus $BV$
observables, that is, there exists $C, \tau > 0$ such that for all
$k \in \naturals$ and for all functions
$\varphi_1, \varphi_2 \colon [0,1] \to \reals$ with $\varphi_1 \in BV$
and $\varphi_2 \in L^{1}$,
\begin{equation}\label{eq:doc:interval}
  \biggl| \int_{[0,1]} \varphi_1 \varphi_2 \circ T^{k} \diff\mu -
  \int_{[0,1]} \varphi_1 \diff\mu \int_{[0,1]} \varphi_2 \diff\mu
  \biggr| \leq C \bvnorm{\varphi_1} \lnorm{\varphi_2}{1} e^{-\tau k}.
\end{equation}
By iterating \eqref{eq:doc:interval}, we obtain that
\eqref{eq:doc:total} in Assumption~\ref{assumption:doc} holds for
$\varphi_1, \varphi_2 \in BV$ and $\varphi_3 \in L^{1}$.
By the assumptions on $T$, we have that \eqref{eq:doc:partial} holds
for $\sigma < \tau / \log (\bvnorm{T})$ and Lipschitz continuous
$\varphi_1, \varphi_2$.
Indeed, by the above assumptions, we have that
$\bvnorm{T^{k}} \leq \bvnorm{T}^{k}$. Hence, for a Lipschitz continuous
observable $\varphi_2$, we have that $\varphi_2 \circ T^{k} \in BV$ and
that
$\bvnorm{\varphi_2 \circ T^{k}}
\leq \lipnorm{\varphi_2} \bvnorm{T}^{k}$.
Thus for $k \leq \sigma l$,
\[
\begin{split}
  \biggl| \int_{[0,1]} \varphi_1 \varphi_2 \circ T^{k} \varphi_3 \circ
  T^{k+l} \diff\mu\ -& \int_{[0,1]} \varphi_1 \varphi_2 \circ T^{k}
  \diff\mu \int_{[0,1]} \varphi_3 \diff\mu \biggr| \\
  &\leq C \bvnorm{\varphi_1 \varphi_2 \circ T^{k}} \lnorm{\varphi_3}{1}
  e^{-\tau (k + l)} \\
  &\leq C \bvnorm{\varphi_1} \bvnorm{\varphi_2} \lnorm{\varphi_3}{1}
  e^{k(\log \bvnorm{T})} e^{-\tau l} \\
  &\leq C \lipnorm{\varphi_1} \lipnorm{\varphi_2} \lipnorm{\varphi_3}
  e^{-\tau' l},
\end{split}
\]
where $\tau' = \tau - \sigma \log (\bvnorm{T})$.
By replacing $\tau$ by $\tau'$, we obtain \eqref{eq:doc:partial}.

Therefore, the conclusions of Lemma~\ref{lem:estimates:general} and
Lemma~\ref{lem:largel} still hold.
We point out that as a consequence of
Lemma~\ref{lem:estimates:general}, there exists $C > 0$ such that
\begin{equation}
  \label{eq:muEk-onedim}
  C^{-1} M_k \leq \mu(E_k) \leq C M_k
\end{equation}
since we may assume without loss of generality that
$M_k \geq C e^{-\tau k}$.

Having established that every other assumption holds, we only need to
verify that Assumption~\ref{assumption:srt} holds in order
to apply Theorem~\ref{thm:dynamical}.

\begin{lemma}
  \label{lem:srt:interval}
  Under the above assumptions on $([0,1], T, \mu)$, there exists
  constants $C > 0$ and $\lambda > 1$ such that
  \[
    \mu (E_k \cap E_{k+l}) \leq C \mu (E_k) \mu (E_{k+l}) + C
    \mu(E_{k+l}) \lambda^{-l}.
  \]
\end{lemma}

\begin{proof}
  We start by describing the geometric structure of the set
  $E_k$. For each cylinder $C_{j_1,\ldots,j_k}$ there is an
  interval $I_{j_1,\ldots,j_k} \subset C_{j_1,\ldots,j_k}$ such
  that $I_{j_1,\ldots,j_k} = E_k \cap C_{j_1,\ldots,j_k}$. The
  interval $I_{j_1,\ldots,j_k}$ contains exactly one point of
  period $k$. We shall first prove that
  \begin{equation}
    \label{eq:3M}
    \mu (T^k I_{j_1,\ldots,j_k}) \leq 3 M_k.
  \end{equation}

  Let $I_{j_1,\ldots,j_k} = (a,b)$ and suppose that $p \in (a,b)$
  is the above mentioned periodic point of period $k$. We prove
  that
  \begin{equation}
    \label{eq:threelittleballs}
    T^k I_{j_1,\ldots,j_k} \subset B (a, r_k (a))
    \cup B (p, r_k (p)) \cup B (b, r_k (b)),
  \end{equation}
  which implies that $\mu (T^k (a,b)) \leq 3 M_k$, since by
  definition, the measure of each of the three balls is $M_k$.

  We consider the endpoint $a$. All following estimates are also
  valid for the other endpoint $b$. Since $a$ is an end-point of
  $(a,b)$ we have that $d (T^k a, a) = r_k (a)$.

  Suppose first that $T' > 0$ on $(a,b)$. Then $T^{k}a < a < p$ and
  \begin{equation}
    \label{eq:distance}
    d (T^k a, p) = d (T^k a, a) + d (a,p) = r_k (a) + d(a,p)
  \end{equation}
  and
  \[
    d (T^k a, p) = |(T^k)' (\xi)| d (a,p),
  \]
  for some point $\xi$ between $a$ and $p$. A consequence is that
  \[
    (|(T^k)' (\xi)| - 1) d (a,p) = r_k (a) \leq r_k (p) + d(a,p),
  \]
  since $r_k$ is 1-Lipschitz. Hence
  \[
    (|(T^k)' (\xi)| - 2) d (a,p) \leq r_k (p),
  \]
  and since $|(T^k)' (\xi)| \geq 3$ we conclude that
  $d (a,p) \leq r_k (p)$. By the equality \eqref{eq:distance} we
  therefore have that
  \[
    d (T^k a, p) \leq r_k (a) + r_k(p),
  \]
  \begin{figure}[ht]
    \centering
    \includegraphics{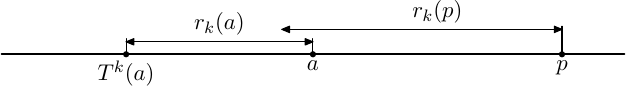}
    \caption{The position of $T^{k}a$ relative to $a$ and $p$ when
    $T' > 0$.}
    \label{fig:balls:1}
  \end{figure}
  which implies that the interval between $T^k a$ and $p$ is
  contained in the two balls $B (a, r_k(a))$ and $B (p, r_k (p))$ (see
  Figure~\ref{fig:balls:1}).
  Similarly, we have that the interval between $T^k (b)$ and $p$ is
  contained in $B (b, r_k(b))$ and $B (p, r_k (p))$. This establishes
  \eqref{eq:threelittleballs} and hence also \eqref{eq:3M}.

  Suppose now that $T' < 0$ on $[a,b]$. Then $a < p < T^{k}a$. It must
  be the case that $r_k(p) > d(p,T^{k}a)$ since otherwise $B(p,r_k(p))$ 
  is a subset of $B(a,r_k(a))$ contradicting the definition of $r_k$.
  This implies that $(p,T^{k}a)$ is contained in $B(p,r_k(p))$.
  Similarly $(T^{k}b,p)$ is also contained in $B(p,r_k(p))$. Thus,
  $T^{k}(a,b)$ is contained in $B(p,r_k(p))$ and that 
  $\mu(T^{k}(a,b)) \leq M_k < 3M_k$ (see Figure~\ref{fig:balls:2}).

  \begin{figure}[ht]
    \centering
    \includegraphics{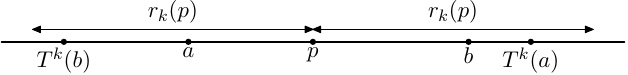}
    \caption{The position of $T^{k}a$ and $T^{k}b$ relative to $a$, $b$
    and $p$ when $T' < 0$.}
    \label{fig:balls:2}
  \end{figure}

  The measure $\mu$ has a density $h$ with respect to a conformal
  measure $\nu$. The density $h$ is of bounded variation, and is
  therefore bounded, and $h$ is also bounded away from zero. That
  $\nu$ is conformal means that if $A \subset C_{j_1,\ldots,j_k}$, then
  \[
    \nu (T^k A) = \int_A e^{S_k \phi} \diff \nu.
  \]
  Since
  \[
  \begin{split}
    M_k &\geqs \mu(T^{k}I_{j_1,\dotsc,j_k}) \\
    &\geqs \nu(T^{k}I_{j_1,\dotsc,j_k}) \\
    &\geqs \nu(T^{k}I_{j_1,\dotsc,j_k}) \\
    &\geqs \int_{I_{j_1,\dotsc,j_k}} e^{S_k\phi} \diff\nu \\
    &\geqs \mu(C_{j_1,\dotsc,j_k})^{-1} \nu(I_{j_1,\dotsc,j_k}) \\
    &\geqs \mu(C_{j_1,\dotsc,j_k})^{-1} \mu(I_{j_1,\dotsc,j_k}),
  \end{split}
  \]
  we have
  \begin{equation}
    \label{eq:measureincylinder}
    \mu (I_{j_1,\ldots,j_k}) \leq C M_k \mu(C_{j_1,\ldots,j_k}).
  \end{equation}
  
  We now consider the intersection $E_k \cap E_{k+l}$. Consider a
  cylinder $C_{j_1,\ldots,j_k}$ and its subinterval
  $I_{j_1,\ldots,j_k}$. We write the intersection
  $I_{j_1,\ldots,j_k} \cap E_{k+l}$ as
  \[
    I_{j_1,\ldots,j_k} \cap E_{k+l} = K_1 \cup K_2 \cup \bigcup
    J_j,
  \]
  where the sets $J_j$ are those $I_{j_1,\ldots,j_{k+l}}$ for
  which their parent cylinder $C_{j_1,\ldots,j_{k+l}}$ is a
  subset of $I_{j_1,\ldots,j_k}$. The sets $K_1$ and $K_2$ consist of
  the remaining parts that are not included in the sets $J_j$ and may
  be empty.

  By \eqref{eq:measureincylinder}, we have
  \[
    \mu \Bigl( \bigcup J_j \Bigr) \leq C M_{k+l} \sum \mu
    (C_{j_1,\ldots,j_{k+l}}) \leq C M_{k+l} \mu
    (I_{j_1,\ldots,j_k}).
  \]
  Using \eqref{eq:quasiBernoulli} and \eqref{eq:cylinderdecay}, we
  have
  \[
  \begin{split}
    \mu (K_1) &\leq M_{k+l} \mu(C_{j_1,\dotsc,j_{k+l}}) \\
    &\leq C M_{k+l} \mu(C_{j_1,\dotsc,j_{k}})
    \mu(C_{j_{k+1},\dotsc,j_{k+l}}) \\
    &\leq C M_{k+l} \lambda^{-l} \mu(C_{j_1,\ldots,j_k}),
  \end{split}
  \]
  and similarly for $K_2$.

  Summing over $I_{j_1,\ldots,j_k}$ we obtain that
  \[
    \mu (E_k \cap E_{k+l}) \leq C \mu (E_k) M_{k+l} + C M_{k+l}
    \lambda^{-l}.
  \]
  By \eqref{eq:muEk-onedim} this implies that
  \[
    \mu (E_k \cap E_{k+l}) \leq C \mu (E_k) \mu (E_{k+l}) + C
    \mu(E_{k+l}) \lambda^{-l}. \qedhere
  \]
\end{proof}

We now prove Proposition~\ref{prop:interval}.
Recall that we assume that $(M_k)$ is a monotonically decreasing
sequence converging to $0$ such that
$M_k \leq C (\log k)^{-2}$ for some $C > 0$.

\begin{proof}[Proof of Proposition~\ref{prop:interval}]
  Assume first that $\sum_{k=1}^\infty M_k < \infty$. In this
  case it is known \cite[Theorem~C]{KKP} that
  $\mu (\limsup E_k) = 0$. The statement of
  Proposition~\ref{prop:interval} is therefore true in this case.

  We now assume that $\sum_{k=1}^\infty M_k = \infty$.  As
  previously mentioned, Assumptions~\ref{assumption:doc},
  \ref{assumption:frostman} and \ref{assumption:thinAnnuli} are
  satisfied.  Furthermore, since $(M_k)$ is monotonically
  decreasing, using \eqref{eq:Ekmeasure:general} yields that
  $\mu(E_{k+l}) \leq C \mu(E_k)$ for some $C > 0$.  As a result,
  we obtain from Lemma~\ref{lem:srt:interval} that
  \[
    \mu(E_k \cap E_{k+l}) \leq C \mu(E_k)^2 + C \mu(E_k) \lambda^{-l},
  \]
  which is exactly the condition in Assumption~\ref{assumption:srt}
  with $N = 0$, $\delta = 1$ and $\psi(l) = \lambda^{-l}$. Thus,
  $M_k \leq C (\log k)^{-2} = C (\log k)^{-(2+N)/\delta}$ and we find
  that that all the assumptions of Theorem~\ref{thm:dynamical} are
  satisfied. The desired result therefore follows.
\end{proof}

\subsection{Linear Anosov maps on \texorpdfstring{$\torus^2$}{T2}
preserving the Lebesgue measure}
\label{sec:anosov}

In this subsection, we let $A$ be a $2 \times 2$ integer matrix with
$\lvert\det A\rvert = 1$ and no eigenvalues on the unit circle. We
define $T \colon \mathbbm{T}^2 \to \mathbbm{T}^2$ by $Tx = Ax \mod 1$
and let $\leb$ denote the Lebesgue measure, which is ergodic.
We let $\lambda$ denote the eigenvalue of $A$ outside the unit circle.

It is established in \cite{KS} that \eqref{eq:doc:total} in
Assumption~\ref{assumption:doc} holds for Anosov diffeomorphisms. In
fact, since $T$ is Lipschitz continuous, \eqref{eq:doc:total} implies
\eqref{eq:doc:partial} for $\sigma < \tau / \log(\lipnorm{T})$. Indeed,
for $k \leq \sigma l$, we have by decay of correlations that
\[
\begin{split}
  \biggl| \int_{X} \tilde{\varphi}_1 \tilde{\varphi}_2 \circ T^{k}
  \tilde{\varphi}_3 \circ T^{k+l} \diff\mu &- \int_{X} \tilde{\varphi}_1
  \tilde{\varphi}_2 \circ T^{k} \diff\mu \int_{X} \tilde{\varphi}_3
  \diff\mu \biggr| \\
  &\leq C \lipnorm{\tilde{\varphi}_1 \tilde{\varphi}_2 \circ T^{k}}
  \lipnorm{\tilde{\varphi}_3} e^{-\tau (k+l)} \\
  &\leq C \lipnorm{\tilde{\varphi}_1} \lipnorm{\tilde{\varphi}_2}
  \lipnorm{\tilde{\varphi}_3} e^{k(\log\lipnorm{T})} e^{-\tau l} \\
  &\leq C \lipnorm{\tilde{\varphi}_1} \lipnorm{\tilde{\varphi}_2}
  \lipnorm{\tilde{\varphi}_3} e^{-\tau'l}
\end{split}
\]
where $\tau' = \tau - \sigma \log(\lipnorm{T})$.
Replacing $\tau$ with $\tau'$ yields \eqref{eq:doc:partial}.

It is clear that Assumptions~\ref{assumption:frostman} and
\ref{assumption:thinAnnuli} hold for $\leb$.
Thus, in order to apply Theorem~\ref{thm:dynamical} to prove
Proposition~\ref{prop:torus}, we only need to verify that
Assumption~\ref{assumption:srt} holds, which follows from
Lemma~\ref{lem:srt:torus}.
The following lemma will be used to prove Lemma~\ref{lem:srt:torus}.

\begin{lemma}
  \label{lem:separation}
  There is a constant $c_A > 0$ such that for any $l$ and any
  $\rho > 0$, the rectangles centred at $l$-periodic points and
  with side-length $\rho$ in the stable direction and
  $c_A \lambda^{-l} \rho^{-1}$ in the unstable direction are
  pairwise disjoint.
\end{lemma}

\begin{proof}
  We let $\pi \colon \mathbbm{R}^2 \to \mathbbm{T}^2$ be the
  projection $\pi (x) = x \mod 1$. The lattice $L = (A^l -
  I)^{-1} \mathbbm{Z}^2$ is the lattice such that $\pi (L)$ is the
  set of $l$-periodic points, that is $p \in \pi (L)$ if and only
  if $T^l p = p$.

  Around each $P \in L$, we place a rectangle $R(P)$ with
  side-length $\rho$ in the stable direction and
  $c_A \lambda^{-l} \rho^{-1}$ in the unstable direction. Mapping
  every rectangle by $t(x) = (A^l - I)x$, we obtain rectangles
  $t(R(P))$ centred at points of $\mathbbm{Z}^2$ and with
  side-lengths $|\lambda^{-l} - 1| \rho$ and $c_A |\lambda^{l} - 1|
  \lambda^{-l} \rho^{-1}$ in the unstable direction. By
  Liouville's theorem in Diophantine approximation, we conclude
  that if $c_A$ is chosen small enough but positive, then
  $t(R(P)) \cap t(R(Q)) = \emptyset$ if $P \neq Q$. Since $t$ is
  invertible it follows that $R(P) \cap R(Q) = \emptyset$ if $P
  \neq Q$.

  This proves that two rectangles $\pi (R(P))$ and $\pi (R(Q))$
  on the torus $\mathbbm{T}^2$ are either disjoint or one and the
  same rectangle.
\end{proof}

\begin{lemma}
  \label{lem:srt:torus}
  There are constants $C > 0$ and $\lambda > 1$ such that
  \begin{equation}
    \label{eq:Ekmeasure:torus}
    \leb (E_k) = \pi r_k^2
  \end{equation}
  and
  \begin{equation}
    \label{eq:Ekquasiindependence:torus}
    \leb (E_k \cap E_{k+l}) \leq C \leb (E_k) \leb (E_{k+l}) + C
    \lambda^{-l} \sqrt{\leb(E_k) \leb (E_{k+l})}.
  \end{equation}
\end{lemma}

\begin{proof}
  The equality \eqref{eq:Ekmeasure:torus} is by Kirsebom, Kunde and
  Persson \cite[Section~6.2]{KKP}. We now prove the inequality
  \eqref{eq:Ekquasiindependence:torus}.

  We will estimate the measure of $E_k \cap E_{k+l}$.

  The matrix $A$ has one eigenvalue $\lambda$ outside the unit
  circle and one eigenvalue $\pm 1/\lambda$ inside the unit
  circle. We assume here that $\lambda > 1$. Things are the same
  if $\lambda < -1$.

  The set $E_k$ consists of ellipses centred around $k$-periodic
  points. More precisely, around each point $p$ such that
  $T^k p = p$, there is an ellipse with semi axes
  $(\lambda^{-k} - 1)^{-1} r_k$ in the stable direction and
  $(\lambda^{k} - 1)^{-1} r_k$ in the unstable direction. Hence
  the semi axes are not larger than $c r_k$ and
  $c \lambda^{-k} r_k$ for some constant $c$. In total there are
  $\lvert\det (A^k - I)\rvert \sim \lambda^k$ periodic points.

  In $E_k$, we replace each ellipse by a rectangle with sides
  $c r_k$ and $c \lambda^{-k} r_k$, and we denote the obtained
  set by $\hat{E}_k$. We do the same for $E_{k+l}$ and get that
  \[
    E_k \cap E_{k+l} \subset \hat{E}_k \cap \hat{E}_{k+l}.
  \]

  \begin{figure}[ht]
    \centering
    \includegraphics[width=120mm, height=55mm]{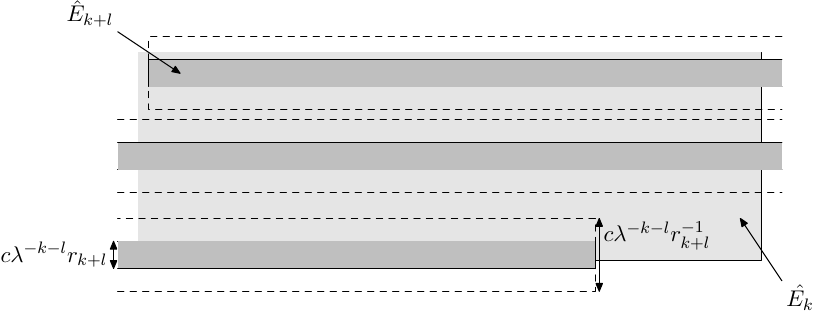}
    \caption{An illustration of the intersection of rectangles in
    $\hat{E}_k$ and $\hat{E}_{k+l}$.}
    \label{fig:rectangles}
  \end{figure}

  The set $\hat{E}_{k+l}$ consists of parallelograms, each
  contained in a rectangle with sides $c r_{k+l}$ and
  $c \lambda^{-k-l} r_{k+l}$. By Lemma~\ref{lem:separation}, in
  the unstable direction, the separation of these rectangles are
  at least $c \lambda^{-k-l} r_{k+l}^{-1}$.  Therefore, for each
  rectangle in $\hat{E}_k$, there are at most
  \[
    \frac{c \lambda^{-k} r_k}{c \lambda^{-k-l} r_{k+l}^{-1}} + 1
    = c \lambda^l r_k r_{k+l} + 1
  \]
  rectangles from $\hat{E}_{k+l}$ that intersect (see
  Figure~\ref{fig:rectangles}). Each such intersection has an
  area of at most $c r_{k+l} \lambda^{-k-l} r_k$.  We can
  therefore bound the area of the intersection
  $\hat{E}_k \cap \hat{E}_{k+l}$ by
  \[
    \leb (\hat{E}_k \cap \hat{E}_{k+l}) \leq c \lambda^k \cdot c
    r_k \lambda^{-k-l} r_{k+l} \cdot (c \lambda^l r_k r_{k+l} +
    1) \leq c r_k^2 r_{k+l}^2 + c \lambda^{-l} r_k r_{k+l}.
  \]
  By \eqref{eq:Ekmeasure:torus}, we therefore have
  \[
    \leb (E_k \cap E_{k+l}) \leq C \leb (E_k) \leb (E_{k+l})
    + C \lambda^{-l} \sqrt{\leb (E_k) \leb (E_{k+l})}.
    \qedhere
  \]
\end{proof}

We now have all the tools needed to prove Proposition~\ref{prop:torus}.
Recall that we assume that $(M_k)$ is a monotonically decreasing
sequence converging to $0$ such that
$M_k \leq C (\log k)^{-2}$ for some $C > 0$.

\begin{proof}[Proof of Proposition~\ref{prop:torus}]
  By \eqref{eq:Ekmeasure:torus}, we have that $m (E_k) = M_k$. It
  follows immediately by the easy part of the Borel--Cantelli
  lemma that if $\sum_{k=1}^\infty M_k < \infty$, then
  $m (\limsup E_k) = 0$, which establishes
  Proposition~\ref{prop:torus} in this case. This case is also
  stated in \cite[Theorem~D]{KKP}. From now on we assume that
  $\sum_{k=1}^\infty M_k = \infty$.
  
  It has been established that Assumptions~\ref{assumption:doc},
  \ref{assumption:frostman} and \ref{assumption:thinAnnuli} hold.

  We now show that Assumption~\ref{assumption:srt} holds with $N=0$ and
  $\delta = 1$. Since $\leb$ is the Lebesgue measure, we have that
  $M_k = \leb(B(x,r_k)) = \pi r_k^2$. Thus, by
  \eqref{eq:Ekmeasure:torus} in Lemma~\ref{lem:srt:torus}, we have that
  $\leb(E_k) = M_k$ and hence that $(\leb(E_k))_k$ is monotonically
  decreasing. As a result, we obtain from
  \eqref{eq:Ekquasiindependence:torus} in Lemma~\ref{lem:srt:torus}
  that
  \[
    \leb(E_k \cap E_{k+l}) \leq C\leb(E_k)^2 + C \leb(E_k) \lambda^{-l},
  \]
  which is exactly the condition in Assumption~\ref{assumption:srt}
  with $\mu=\leb$, $N = 0$, $\delta = 1$ and $\psi(l) = \lambda^{-l}$.
  Finally, since
  $M_k \leq C (\log k)^{-2} = C (\log k)^{-(2+N)/\delta}$, we may apply
  Theorem~\ref{thm:dynamical} and we conclude the proof.
\end{proof}

\subsection*{Acknowledgements}
The authors would like to thank Conze and Le Borgne for communicating
Lemma~\ref{lem:partition} and its proof to us through private
correspondence.


\begin{thebibliography}{00}

\bibitem{BS} L. Barreira, B. Saussol,
  \emph{Hausdorff dimension of measures via Poincar\'e recurrence}, 
  Comm. Math. Phys. \textbf{219} (2001), no.~2, 443–463. 

\bibitem{Boshernitzan} M. D. Boshernitzan,
  \emph{Quantitative recurrence results},
  \textit{Invent. Math.} \textbf{133} (1993), no.~3 617--631.

\bibitem{Bowen} R. Bowen,
  \emph{Equilibrium states and the Ergodic Theory of {{Anosov}}
    Diffeomorphisms} (Number 470 in Lecture Notes in Mathematics),
  2nd rev. edn. {Springer-Verlag}, {Berlin}, 2008.

\bibitem{CK} N. Chernov, D. Kleinbock,
  \emph{Dynamical {{Borel--Cantelli}} lemmas for gibbs measures},
  Israel J. Math. \textbf{122} (2001), 1--27.

\bibitem{CL} J.-P. Conze, S. Le Borgne,
  \emph{Limit law for some modified ergodic sums},
  Stoch. Dyn. \textbf{11} (2011), no.~1, 107--133.

\bibitem{CM} N. Chernov, R. Markarian,
  \emph{Chaotic Billiards},
  Number 127 in Mathematical Surveys and Monographs,
  {American Mathematical Society}, {Providence, RI},
  2006.

\bibitem{Dolgopyat} D. Dolgopyat,
  \emph{Limit theorems for partially hyperbolic systems},
  Trans. Amer. Math. Soc. \textbf{356} (2004), no.~4, 1637--1689.

\bibitem{GK} S. Galatolo, D. H. Kim,
  \emph{The dynamical {{Borel--Cantelli}} lemma and the waiting time
    problems},
  Indag. Math. (N.S.) \textbf{18} (2007), no.~3, 421--434.

\bibitem{GalKoksma} I. S. G\'{a}l and J. F. Koksma, \emph{Sur l'ordre
  de grandeur des fonctions sommables,} Indagationes Math. \textbf{12}
  (1950), 192--207.

\bibitem{GS} A. Galves, B. Schmitt,
  \emph{Inequalities for hitting times in mixing dynamical systems},
  Random Comput. Dynam. \textbf{5} (1997), no.~4, 337--347.

\bibitem{GHN} C. Gupta, M. Holland, M. Nicol,
  \emph{Extreme value theory and return time statistics for dispersing
    billiard maps and flows, Lozi maps and Lorenz-like maps},
  Ergodic Theory Dynam. Systems \textbf{31} (2011), no.~5, 1363--1390.

\bibitem{Harman} G. Harman, \emph{Metric number theory}, London
  Mathematical Society Monographs. New Series, 18, The Clarendon Press,
  Oxford University Press, New York, 1998, ISBN:0-19-850083-1.

\bibitem{HNPV} N. Haydn, M. Nicol, T. Persson, S. Vaienti,
  \emph{A note on {{Borel--Cantelli}} lemmas for non-uniformly
    hyperbolic dynamical systems},
  Ergodic Theory Dynam. Systems \textbf{33} (2013), no.~2, 475--498.

\bibitem{He} Y. He,
  \emph{Quantitative recurrence properties and
    strong dynamical Borel--Cantelli lemma for dynamical systems
    with exponential decay of correlations},
  arXiv:2410.10211.

\bibitem{HLSW} M. Hussain, B. Li, D. Simmons, B. Wang,
  \emph{Dynamical {{Borel--Cantelli}} lemma for recurrence theory},
  Ergodic Theory Dynam. Systems \textbf{42} (2022), no.~6, 1994--2008.

\bibitem{Kim} D. H. Kim,
  \emph{The dynamical Borel--Cantelli lemma for interval maps},
  Discrete Contin. Dyn. Syst. \textbf{17} (2007), no.~4, 891--900.

\bibitem{KKP} M. Kirsebom, Ph. Kunde, T. Persson,
  \emph{On shrinking targets and self-returning points},
  Annali della Scuola Normale Superiore di Pisa, Classe di Scienze,
  Vol. \textbf{24}, (2023) no.~3, 1499--1535.

\bibitem{KZ} D. Kleinbock, J. Zheng,
  \emph{Dynamical {{Borel--Cantelli}} lemma for recurrence under
    {{Lipschitz}} twists},
  Nonlinearity \textbf{36} (2023), no.~2, 1434--1460.

\bibitem{KS} M. Kotani, T. Sunada,
  \emph{The pressure and higher correlations for an {{Anosov}}
    diffeomorphism},
  Ergodic Theory Dynam. Systems \textbf{21} (2001), no.~3, 807--821.

\bibitem{LLSV} J. Levesley, B. Li, D. Simmons, S. Velani,
  \emph{Shrinking targets versus recurrence: the quantitative
    theory},
  arXiv:2410.22993.

\bibitem{LSV} C. Liverani, B. Saussol, S. Vaienti,
  \emph{Conformal measure and decay of correlation for covering
  weighted systems},
  Ergodic Theory Dynam. Systems \textbf{18} (1998), no.~6, 1399--1420.

\bibitem{Pene} F. P\`ene,
  \emph{Multiple decorrelation and rate of convergence in
  multidimensional limit theorems for the Prokhorov metric},
  Ann. Probab. \textbf{32} (2004), no.~3B, 2477--2525.

\bibitem{Persson} T. Persson.
  \emph{A strong {{Borel--Cantelli}} lemma for recurrence},
  Studia Math. \textbf{268} (2023), 75--89.

\bibitem{Philipp} W. Philipp,
  \emph{Some metrical theorems in number theory},
  Pacific J. Math. \textbf{20} (1967),  109--127.

\bibitem{ARS} A. Rodriguez Sponheimer,
  \emph{A recurrence-type strong Borel–Cantelli lemma for {Axiom A}
  diffeomorphisms},
  Ergodic Theory Dynam. Systems (2024), advance online publication.

\bibitem{VGS} V. G. Sprind\v{z}uk,
  \emph{Metric Theory of Diophantine Approximations},
  V. H. Winston \& Sons, Washington, DC, John Wiley \& Sons, New
  York-Toronto-London,
  1979.

\end{thebibliography}
\end{document}